\documentclass[11pt]{amsart}
\usepackage{amsthm}
\theoremstyle{plain}
\newtheorem{thm}{Theorem}[section]
\newtheorem{theorem}[thm]{Theorem}

\theoremstyle{definition}

\newtheorem{definition}[thm]{Definition}

\numberwithin{equation}{section}
 \title[Coordinate-wise Armijo's condition]{Coordinate-wise Armijo's condition}
 \author{Tuyen Trung Truong}
   \address{Department of Mathematics, University of Oslo, Blindern 0851 Oslo, Norway}
  \email{tuyentt@math.uio.no}
    \date{\today}
   \subjclass[2010]{}

\begin{document}
\maketitle
%{\centering\footnotesize To my daughter on her birthday occasion\par}

\begin{abstract}
Let $z=(x,y)$ be coordinates for the product space $\mathbb{R}^{m_1}\times \mathbb{R}^{m_2}$. Let $f:\mathbb{R}^{m_1}\times \mathbb{R}^{m_2}\rightarrow \mathbb{R}$ be a $C^1$ function, and $\nabla f=(\partial _xf,\partial _yf)$ its gradient. Fix $0<\alpha <1$.  For a point $(x,y) \in \mathbb{R}^{m_1}\times \mathbb{R}^{m_2}$, a number $\delta >0$ satisfies Armijo's condition at $(x,y)$ if the following inequality holds:
\begin{eqnarray*}
f(x-\delta \partial _xf,y-\delta \partial _yf)-f(x,y)\leq -\alpha \delta (||\partial _xf||^2+||\partial _yf||^2).  
\end{eqnarray*}
     
When $f(x,y)=f_1(x)+f_2(y)$ is a coordinate-wise sum map, we propose the following {\bf coordinate-wise} Armijo's condition. Fix again $0<\alpha <1$. A pair of positive numbers $\delta _1,\delta _2>0$ satisfies the coordinate-wise variant of Armijo's condition at $(x,y)$ if the following inequality holds: 
\begin{eqnarray*}
[f_1(x-\delta _1\nabla f_1(x))+f_2(y-\delta _2\nabla f_2(y))]-[f_1(x)+f_2(y)]\leq -\alpha (\delta _1||\nabla f_1(x)||^2+\delta _2||\nabla f_2(y)||^2). 
\end{eqnarray*} 

We then extend results in our recent previous results, on Backtracking Gradient Descent and some variants, to this setting. We show by an example the advantage of using coordinate-wise Armijo's condition over the usual Armijo's condition. 
\end{abstract}

%\section{Introduction} 
%\section{Results}

\subsection{Gradient Descent methods and Armijo's condition}

Gradient Descent (GD) methods, invented by Cauchy in 1847, aim to find minima of a $C^1$ function $f:\mathbb{R}^m\rightarrow \mathbb{R}$ by the following iterative procedure
\begin{eqnarray*}
z_{n+1}=z_n-\delta (z_n)\nabla f(z_n),
\end{eqnarray*}
where $\delta (z_n)>0$, the learning rate, must be appropriately chosen. Armijo's condition \cite{armijo} is a well known criterion to find learning rates, which require that
\begin{eqnarray*}
f(z-\delta (z)\nabla f(z))-f(z)\leq -\alpha ||\nabla f(z)||^2,
\end{eqnarray*}
where $0<\alpha <1$ is a given constant. 

{\bf Backtracking GD.} Armijo's condition gives rise to Backtracking GD, which is the following procedure. Given $0<\alpha ,\beta <1$ and $\delta _0>1$. For each $z\in \mathbb{R}^m$ we define $\delta (z)$ to be the largest number among $\{\beta ^n\delta _0:~n=0,1,2,\ldots \}$ which satisfies Armijo's condition. For each initial point $z_0$, we then define inductively the sequence $z_{n+1}=z_n-\delta (z_n)\nabla f(z_n)$. 

There are many other variants of GD methods (such as Momentum, NAG and Adam), the interested reader can consult for example the overview paper \cite{ruder}. Here we discuss only one other version of GD which is defined as follows. 

{\bf Standard GD.} Fix $\delta _0>0$. For each initial point, we define inductively the sequence $z_{n+1}=z_n-\delta _0\nabla f(z_n)$.

{\bf Convention.} In the literature of Optimization and Deep Neural Networks, there are many results where "convergence" for an iterative method is meant to be that either $\lim _{n\rightarrow\infty}f(z_n)=-\infty$, or $\lim _{n\rightarrow\infty}\nabla f(z_n)=0$ for the sequences $\{z_n\}$ constructed by that iterative method. In this paper, however, we use the more strict mathematical notion of convergence, requiring that either $\lim _{n\rightarrow\infty}||z_n||=\infty$ or there is $z_{\infty}$ so that $\lim _{n\rightarrow \infty}z_n=z$ and $\nabla f(z_{\infty})=0$. In the appendix to \cite{truong}, we argued why the more strict mathematical notion of convergence is better for applications of GD methods in Deep Neural Networks. We note that many people, including experts, are confused between these two notions.  

Standard GD is used often in practice because of its simplicity. If the gradient $\nabla f$ is globally Lipschitz continuous with Lipschitz constant $L$, then convergence for  $\{z_n\}$ in Standard GD has been proven under some extra assumptions (see \cite{lange}): the learning rate $\delta $ is $<1/(2L)$, $f$ has compact sublevels (that is, all the sets $\{x:~f(x)\leq b\}$ for $b\in \mathbb{R}$ are compact), and the set of critical points of $f$ is bounded and isolated (which in effect implies that $f$ has only a finite number of critical points). If the sequence $\{z_n\}$ converges, and the initial point $z_0$ is outside of a pre-determined set $\mathcal{E}$ of Lebesgue measure zero, then the limit point cannot be a generalised saddle point, see \cite{lee-simchowitz-jordan-recht, panageas-piliouras}. Here, by a generalised saddle point, we mean a critical point $z_{\infty}$ near which $f$ is $C^2$  and the Hessian $\nabla ^2f(z_{\infty})$ has at least one {\bf negative} eigenvalue. However, in practical applications such as Deep Neural Networks (DNN), there are many cases where these assumptions are not satisfied or not easy to check. Hence, a common practice in the DNN community is to try to do many seek-and-try attempts (manual fine-tuning or grid search) to find a learning rate which allow to achieve good performance. This practice is however time-consuming, and can lead to instability and non-reproducibility, see the appendix in \cite{truong} for a detailed analysis.   

Backtracking GD is so far the best theoretically guaranteed method among Gradient Descent methods and other iterative methods such as Newton's, and also performs very well in practice. In fact, it is well known that for a sequence $\{z_n\}$ constructed by Backtracking GD method, if $z_{\infty}$ is a cluster point of the sequence, then $\nabla f(z_{\infty}) =0$. When $f$ is {\bf real analytic}, \cite{absil-mahony-andrews} shows that $\{z_n\}$ always converge. In the case of a general $C^1$ function $f$, it was shown in \cite{truong-nguyen} that if $f$ has at most countably many critical points (which is satisfied by Morse functions, in turn consists of an open and dense subset of the set of all functions), then $\{z_n\}$ converges.  In that paper, we also defined Backtracking versions for Momentum and NAG, and proved corresponding convergence results. We also defined a version called Two-way Backtracking GD, based on the observation that in the long run the learning rates used in Backtracking GD will belong to a finite set and hence we can start the search for learning rate $\delta (x_{n+1})$ not at $\delta _0$ but at $\delta (x_n)$, and multiply or divide by $\beta$ depending on whether Armijo's condition is satisfied for $\delta (x_n)$ or not. This helps to save time and make performing Backtracking GD in practical examples not much more costly than Standard GD, when one takes into account the practice of manual fine-tuning or grid search of learning rates mentioned in the above paragraph. We have implemented algorithms using Two-way Backtracking GD and some other refinements in DNN, with performance on the image dataset CIFAR10 better than other state-of-the art optmization methods such as Momentum, NAG and Adam. ({\bf Important remark:} Here the comparison was done using same DNN architectures.) Our algorithms also work well across different sizes of mini-batches and DNN architectures, and find learning rates automatically (in particularly, avoid completely manual fine-tuning or grid search of learning rates). Source codes for the algorithms can be found in \cite{mbtoptimizer}. A result on avoidance of saddle points, with conclusions weaker than \cite{lee-simchowitz-jordan-recht, panageas-piliouras}, for all $C^1$ functions is also given in \cite{truong-nguyen}.  When the gradient $\nabla f$ is moreover {\bf locally} Lipschitz continuous, we defined  in \cite{truong} a continuous version of Backtracking GD and a new discrete version of Backtracking GD, and prove convergence to minima for these methods, under more general assumptions than in \cite{lee-simchowitz-jordan-recht, panageas-piliouras}.

\subsection{Coordinate-wise Armijo's condition and Backtracking GD}

The purpose of the current paper is to improve Armijo's condition to adapt better to special cases of cost functions $f$. From now on, we consider exclusively the following {\bf special case}: $\mathbb{R}^m=\mathbb{R}^{m_1}\times \mathbb{R}^{m_2}$ is a product space with coordinates $z=(x,y)$, and $f(x,y)=f_1(x)+f_2(y)$ is a sum. In this case, besides the standard Armijo's condition for $f$, we can also make use of the Armijo's conditions for $f_1$ and $f_2$ separately. We have thus the following definitions.

{\bf Coordinate-wise Armijo's condition.} Fix $0<\alpha <1$ A pair $\delta _1,\delta _2>0$ satisfies Coordinate-wise Armijo's condition at $z=(x,y)$ if 
\begin{eqnarray*}
f(x-\delta _1\nabla f_1(x),y-\delta _2\nabla f_2(y))-f(x,y)\leq -\alpha [ \delta _1||\nabla f_1(x)||^2+\delta _2||\nabla f_2(y)||^2 ].
\end{eqnarray*}

\begin{definition} 
{\bf Coordinate-wise Backtracking GD.} Fix $0<\alpha ,\beta <1$ and $\delta _0>0$. We define for each $z=(x,y)$ the pair $\delta _1(x),\delta _2(y)$ given by the usual Backtracking GD procedure for $f_1(x)$ and $f_2(y)$ respectively. We then define for each initial point $z_0=(x_0,y_0)$ iteratively the sequence 
\begin{eqnarray*}
z_{n+1}=(x_n-\delta _1(x)\nabla f_1(x),y-\delta _2(y)\nabla f_2(y)).  
\end{eqnarray*}  
\label{Definition1}\end{definition}

{\bf Remark.} If $\delta (x,y)$ is the number given by the usual Backtracking GD for $f$ itself, then $\delta (x,y)\leq \max \{\delta _1(x),\delta _2(y)\}$. 

Inspired by the corresponding definition in \cite{truong}, we also introduce the following new discrete, coordinate-wise version of Backtracking GD. 

\begin{definition}
{\bf Coordinate-wise Backtracking GD-New.} Assume that $f_1$ and $f_2$ satisfy the assumptions in Definition 1.2 in \cite{truong} (that is, Backtracking GD-New). That is, there are continuous functions $r_1,L_1:\mathbb{R}^{m_1}\rightarrow (0,\infty)$ and $r_2,L_2:\mathbb{R}^{m_2}\rightarrow (0,\infty )$ so that  for every $x\in \mathbb{R}^{m_1}$ and $y\in \mathbb{R}^{m_2}$, then $f_1$ is Lipschitz continuous on $B(x,r_1(x))$ with Lipschitz constant $L_1(x)$, and $f_2$ is Lipschitz continuous on $B(y,r_2(y))$ with Lipschitz constant $L_2(y)$.

Fix $0<\alpha ,\beta <1$ and $\delta _0>0$. We define for each $z=(x,y)$ the pair $\delta _1(x),\delta _2(y)$ as follows: $\delta _1(x)$ is the largest number $\delta $ among $\{\beta ^n\delta _0:~n=0,1,2\ldots \}$ satisfying the following assumptions:
\begin{eqnarray*}
&&\delta < \alpha /L_1(x),\\
&&\delta ||\nabla f_1(x)|| < r_1(x).
\end{eqnarray*}
We define $\delta _2(y)$ similarly. 

The Coordinate-wise Backtracking GD-New procedure is as follows. For each initial point $z_0=(x_0,y_0)$, we define iteratively 
\begin{eqnarray*}
z_{n+1}=(x_n-\delta _1(x_n)\nabla f_1(x_n), y_n-\delta _2(y_n)\nabla f_2(y_n)).
\end{eqnarray*}
\label{Definition2}\end{definition}

\subsection{Main results}

Coordinate-wise Backtracking versions can also be defined for the other variants in \cite{truong-nguyen, truong}, and all the main results in these two papers can be extended to the coordinate-wise version. However, to simplify the presentation, we state the main results of this paper only for the two versions defined in the previous subsections. 

\begin{theorem}
Let $f:\mathbb{R}^{m_1}\times \mathbb{R}^{m_2}$ be a $C^1$ function of the form $f(x,y)=f_1(x)+f_2(y)$. For each initial point $z_0=(x_0,y_0)$, we define the sequence $\{z_{n}\}$ by the Coordinate-wise Armijo's condition. Then we have the following: 

1) Any cluster point of $\{z_n\}$ is a critical point of $f$. 

2) Either $\lim _{n\rightarrow\infty}||z_n||=\infty$ or $\lim _{n\rightarrow\infty}||z_{n+1}-z_n||=\infty$.

3) If $f$ has at most countably many critical points, then either $\{z_n\}$ converges to a point $z_{\infty}$ for which $\nabla f(z_{\infty})=0$, or $\lim _{n\rightarrow \infty}||z_n||=\infty$. 

4) In general, if every connected component of the critical point set of $f$ is compact, then the set of cluster points of $\{z_n\}$ is connected. 
 
\label{TheoremMain}\end{theorem}
\begin{proof}
We simply apply the corresponding results in \cite{truong-nguyen} separately for $f_1$ and $f_2$.
\end{proof}

\begin{theorem}
Let $f$ be a $C^{1}$ function which satisfies the condition in Definition \ref{Definition2}. Assume moreover that $\nabla f$ is $C^2$ near its generalised saddle points. Choose $0<\delta _0$ and $0<\alpha ,\beta <1$. For any initial points $z_0=(x_0,y_0)$, we construct the sequence $\{z_n\}$ as in Definition \ref{Definition2}. Then: 

(i) For every $z_0$, the sequence $\{z_{n}\}$ either satisfies $\lim _{n\rightarrow\infty}||z_{n+1}-z_n||=0$ or $
\lim _{n\rightarrow\infty}||z_n||=\infty$. Each cluster point of $\{z_n\}$ is a critical point of $f$. If moreover, $f$ has at most countably many critical points, then $\{z_n\}$ either converges to a critical point of $f$ or $\lim _{n\rightarrow\infty}||z_n||=\infty$. 

(ii) For {\bf random choices} of $\delta _0, \alpha $ and $\beta$, there is a set $\mathcal{E}_1$ of Lebesgue measure $0$ so that for all $z_0\notin  \mathcal{E}_1$, any cluster point of the sequence $\{z_{n}\}$ cannot be an isolated generalised saddle point. 

(iii) Assume that there is $L_0>0$ so that if $z=(x,y)$ is a {\bf non-isolated} generalised saddle point of $f$, then $\max\{L_1(x),L_2(y)\}\leq L_0$. Then, for {\bf random choices} of $\delta _0, \alpha $ and $\beta$ with $ \alpha /L_0 \notin \{\beta ^n\delta _0:~n=0,1,2,\ldots \}$, there is a set $\mathcal{E}_2$ of Lebesgue measure $0$ so that for all $z_0\notin  \mathcal{E}_2$, if the sequence $\{z'_n\}$  (constructed in Definition \ref{Definition2}, {\bf but} with $L_1(x)$ replaced by $\max\{L_1(x),L_0\}$, and $L_2(y)$ replaced by $\max\{L_2(y),L_0\}$) converges, then the limit point cannot be a generalised saddle point.

\label{Theorem1}\end{theorem}
\begin{proof}
We simply apply Theorem 1.3 in \cite{truong} separately to $f_1$ and $f_2$.   
\end{proof}

\subsection{An example}

 Here we provide an example illustrating the advantage of using coordinate-wise Armijo's condition over using the usual Armijo's condition. 

Let $f(x,y)=x^3\sin (1/x) + y^3\sin (1/y)$. It is $C^1$, but is not locally Lipschitz continuous near $\{x=0\}\cup \{y=0\}$. It is neither convex nor real analytic. It has countably many critical points and countably many generalised saddle points. In fact, one can check that a critical point of $f$ is either $(0,0)$, or a local minimum, or a generalised saddle point. Moreover, $(0,0)$ is the only non-isolated critical point of $f$. 

This map has the special form which Coordinate-wise Armijo's condition can be applied, $f(x,y)=g(x)+g(y)$ where $g(t)=t^3\sin (1/t)$. We can compute that $g'(t)=3t^2\sin (1/t) - t\cos (1/t)$. The second derivative $g"$ is continuous in $\mathbb{R}\backslash \{0\}$, and does not exist at $0$. A rough estimate  gives us $|g"(t)|\leq 6|t|+4+1/|t|$ when $t\not= 0$, and $\sup _{t\not= 0}|g"(t)|=\infty$. 

This function does not satisfy the assumptions in Definition \ref{Definition2}.  However, using the ideas in Example 1 in \cite{truong}, we can proceed as follows. Choose $r:\mathbb{R}\rightarrow (0,\infty)$ any continuous function such that $r(t)<|t|$ for all $t$, and choose $L(t)=\sup _{B(t,r(t))}||g"(t)||$. Fix $0<\alpha ,\beta <1$ and $\delta _0>0$. For $t\not=0$, we define $\delta (t)$ as in Definition \ref{Definition2} (or, the same, as in Definition 1.3 in \cite{truong}). Now the following claim is clear, given that $L(t)$ is continuos. 

{\bf Claim 1.} If $K\subset \mathbb{R}$ is a compact subset so that $0\notin K$, then $\sup _{t\in K}\delta (t)>0$. 

For each initial point $t_0$, we define iteratively $t_{n+1}=t_n-\delta (t_n)\nabla g(t_n)$. We have the following. 

{\bf Claim 2.} Any cluster point of $\{t_n\}$ is a critical point of $g$. 
\begin{proof}
If the cluster point is $0$, then there is nothing to prove. Otherwise, we can use Claim 1 and arguments in \cite{bertsekas} to obtain the assertion. 
\end{proof}

Similarly as in \cite{truong-nguyen} and \cite{truong}, we obtain the following result, by noting that $f$ has compact supports. 

{\bf Claim 3.} For every initial point $t_0$, the sequence $\{t_n\}$ converges to a critical point of $g$. 

We now have the following theorem about convergence to minima for $g$. 

{\bf Claim 4.} Assume that $\alpha, \beta ,\delta _0$ are randomly chosen. There is a set $\mathcal{E}\subset \mathbb{R}$ so that if $t_0\in \mathbb{R}\backslash \mathcal{E}$ then $\{t_n\}$ converges either  to $0$ or a local minimum of $g$. 
\begin{proof}
By using Theorem \ref{Theorem1} and Claim 3, also the special properties of the function $f$, we obtain that under the assumptions in Claim 4, the sequence $\{t_n\}$ converges either to $0$ or a local minimum of $g$. 
\end{proof}

Finally, we obtain the following theorem about convergence to minima for $f(x,y)=g(x)+g(y)$. 

{\bf Claim 5.} Assume that $\alpha, \beta ,\delta _0$ and $\mathcal{E}$ is chosen as in Claim 4. Then for all $z_0=(x_0,y_0)\notin (\mathcal{E}\times \mathbb{R}\cup \mathbb{R}\times \mathcal{E})$, the sequence $z_{n+1}=(x_n-\delta (x_n)g'(x_n), y_n-\delta (y_n)g'(y_n))$ (constructed as above) converges, and the limit point is  either $(0,0)$ or a local minimum of $f$. 
\begin{proof}
The proof follows immediately from Claim 4. 
\end{proof}

{\bf Remarks.} Note that the set $\mathcal{E}\times \mathbb{R}\cup \mathbb{R}\times \mathcal{E}$ has Lebesgue measure 0. Hence, here we proved a stronger than that stated in Example 1 in \cite{truong}. Moreover, there is a mistake with the proof given for Example 1 in \cite{truong}, that is the function $\delta (x,y)$ does not satisfy the property needed to show that any cluster point of $\{z_n\}$ therein is a critical point of $f$. The point is that if $K$ is a compact subset of $\mathbb{R}^2$ on which $\nabla f$ is nowhere $0$, then $\inf _K\delta (x,y)>0$ only if $K$ is a subset of $\mathbb{R}^2\backslash (\{x=0\}\cup \{y=0\})$. On the other hand, proof of part 1 of Theorem 2.1 in \cite{truong-nguyen} still goes through, and hence we can still prove that either $\lim _{n\rightarrow \infty}||z_{n+1}-z_n||=0$ or $\lim _{n\rightarrow\infty}||z_n||=\infty$. Therefore, the following correction of statements concerning Example 1 in \cite{truong} is valid. 

{\bf Claim 6.} Let the sequence $\{z_n\}$ be constructed as in Example 1 in \cite{truong}.  Then there is a set $\mathcal{E}'\subset \mathbb{R}^2$ of Lebesgue measure $0$ so that one of the following two statements is valid:  

(1) The sequence $\{z_n\}$ converges to a critical point of $f$ contained in $\mathbb{R}^2\backslash (\{x=0\}\cup \{y=0\})$.

(2) The cluster set of $\{z_n\}$ is a connected subset of $\{x=0\}\cup \{y=0\}$.

\subsection{Conclusions}
 
In this paper we defined the coordinate-wise version of Armijo's condition for functions of the special form: $f(x,y)=f_1(x)+f_2(y)$. We then used this to define coordinate-wise versions  of Backtracking GD and other variants defined in \cite{truong-nguyen, truong}, and we proved the analogs of these two papers for the new coordinate-wise versions. We then applied the results to obtain a result on convergence to minima for $f(x,y)=x^3\sin (1/x) + y^3\sin (1/y)$. The result obtained here is stronger than that obtained in Example 1 in \cite{truong}, and currently cannot be treated by methods of other authors. Besides open questions listed in \cite{truong-nguyen, truong} for general Backtracking GD and variants, we think that the following three questions tailored for the coordinate-wise versions are worth exploring in the future. 

{\bf Question 1.} In the example in previous subsection, can we also avoid the point $(0,0)$? An affirmative answer can advance techniques to deal with iterative methods to find minima for more general $C^1$ functions, and will most likely involve new insights in Dynamical Systems. 

{\bf Question 2.}  Is there a coordinate version for Armijo's condition, which applies for general functions (and not just those of the special type considered in this paper) and still allows ones to obtain strong results as in Theorems \ref{TheoremMain} and \ref{Theorem1} here, as well as other main results in \cite{truong-nguyen, truong}? 

{\bf Question 3.} For functions of the special form $f(x,y)=f_1(x)+f_2(y)$, is it generally better to use Coordinate-wise Armijo's condition than the usual Armijo's conditions? One should take into account that in the remark after Definition \ref{Definition1}, even though we always have $\delta (x,y)\leq \max\{\delta _1(x),\delta _2(y)\}$, it may happen that $\delta (x,y)>\min\{\delta _1(x),\delta _2(y)\}$. The other point is that when using Coordinate-wise Armijo's condition, it may make it easier for the sequence $\{z_n\}$ to converge to infinity instead of a finite point.

\end{document}